\documentclass[a4paper,10pt,oneside]{amsart}
\usepackage{algorithmic}
\usepackage{algorithm}
\usepackage{mathtools}
\usepackage{pgfplots}
\usetikzlibrary{ decorations.markings}

\newtheorem{theorem}{Theorem}

\newtheorem*{theorem*}{Theorem}
\newtheorem{definition}{Definition}

\newtheorem*{lemma*}{Lemma}

\newtheorem*{assumption*}{Assmuption}

\title[Turnpikes, entry and leaving arcs]{Entry and leaving arcs of turnpikes: their exact computation in the calculus of variations}
\usepackage{enumerate}

\definecolor{nuevo}{RGB}{0,0,120}

\newboolean{pedro}
\setboolean{pedro}{true}
\newcommand{\pedro}[1]{\ifthenelse{\boolean{pedro}}{\color{black!40!red}\marginpar{$\star$ #1}\setboolean{pedro}{false}}{\color{black}\setboolean{pedro}{true}}}

\begin{document}
\author{L. Bayón, P. Fortuny Ayuso, J.M. Grau, M.M. Ruiz}
\address{All authors' address: Dpto. de Matemáticas, Universidad de Oviedo. Spain.}
\email{bayon@uniovi.es}
\email{fortunypedro@uniovi.es}
\email{grau.ribas@uniovi.es}
\email{mruiz@uniovi.es}

\maketitle
\begin{abstract}
  We settle the question of how to compute the entry and leaving arcs for turnpikes in autonomous variational problems, in the one-dimensional case using the phase space of the vector field associated to the Euler equation, and the initial/final and/or the transversality condition. The results hinge on the realization that extremals are the contours of a well-known function and that that the transversality condition is (generically) a curve. An approximation algorithm is presented and an example included for completeness.
\end{abstract}

\section{Introduction}

The idea of \emph{turnpike} in the Calculus of Variations or in Optimal Control describes
the (usual) phenomenon which takes place when, in problems with finite but arbitrarily
large time, the optimal solutions \emph{spend most of their time} near a specific
point. Even more, these solutions tend to be composed of three parts: an entry arc, the
turnpike arc, and the leaving arc. The first and last ones are \emph{transitory} arcs
which take little time of the solution, whereas the middle arc (the turnpike) is a long
arc which is essentially stationary, and tends to be exponentially near an equilibrium
(see \cite{Trelat Zuazua 2015}). Roughly speaking, in the long term, approximate solutions
to problems having a turnpike are determined essentially by the integrand function of the
objective functional, and are ---again, essentially--- independent of their endpoints and
time interval.

Although the first works on the topic investigated specific problems arising in the
context of economics and econometrics \cite{Samuelson 1958}, \cite{McKenzie 1976}, today
the turnpike property has become of interest in other areas \cite{Faulwasser 2014, Grune
  2016}. Recent studies have proposed its use in applications varying from
membrane-filtration systems \cite{Kalboussi Rapaport 2018} to control of chemical reactors
with uncertain models \cite{Faulwasser 2019} or shape optimization in aircraft design
\cite{Lance Trelat Zuazua 2020}.

The property has also been noticed in Optimal Control Problems of almost any type:
with/without terminal constraints \cite{Trelat Zuazua 2015}, \cite{Clarke 2013};
with/without discounted cost functionals \cite{Zaslavski 2014}, \cite{Gurman 2004};
discrete-time problems with constraints \cite{Grune 2016}, \cite{Damm 2014}; and in a
continuous-time problems without constraints \cite{Trelat Zuazua 2015}... Of course, no
work on the turnpike property can omit referencing Zaslavski's exhaustive studies, whose
results and complete references are collected in \cite{Zaslavski Springer 2006, Zaslavski
  2015 Springer}.

From a practical point of view, the interest of the turnpike phenomenon arises from the
fact that under this condition, the computation of (approximate) optimal trajectories in
all areas of optimal control and variational problems becomes trivial for long enough
time spans. In this sense, one of the first applications is \cite{Anderson Kokotovic 1987}, where a
time-invariant linear quadratic optimal control problem is studied. They prove that the
optimal trajectory is approximately composed of two solutions of two infinite-horizon
optimal control problems. With $x(0)$ fixed, the solution for the interval $(0,+\infty )$
defines the part of the trajectory for the original problem from $x=0$ to the
turnpike. With $x(T)$ fixed, the solution for the interval $(-\infty ,T)$ defines the part
of the trajectory of the original problem from the turnpike to $t=T$. The two parts are
then pieced together and exhibit a similar transient behavior. Their approach is
elementary and points out very clearly that the hyperbolicity phenomenon is the heart of
the turnpike property.

Recently, in \cite{Trelat Zuazua 2015}, the authors investigate the relation between the
turnpike propierty and numerical methods (direct and indirect) for a general nonlinear
optimal control problem, without any specific assumption, and for very general terminal
conditions. In the context of turnpike theorem, they provide a new method to ensure the
successful initialization of numerical methods. Assuming that the Pontryagin maximum
principle has been applied, the usual shooting method can be used. However, this is in
general very hard to initialize. As a solution, they propose a variant: as the extremal is
approximately known along the interval $[\varepsilon ,T-\varepsilon ]$, for some
$\varepsilon >0$ (i.e. the turnpike), but it not at the endpoints $t=0$ and $t=T,$ the
idea is choose some arbitrary point of $[\varepsilon ,T-\varepsilon ]$ (for instance
$t=T/2$), and then integrate backwards over $[0,T/2]$ to get an approximation to $x(0)$,
and forward over $[T/2,T]$ to get the approximation to $x(T)$. The unknown value of
$x(T/2)$ must be adjusted, for instance, through a Newton method, so that transversality
conditions are satisfied.

Even more recently, in the same spirit, in \cite{Faulwasser Grune
2020} the authors use the turnpike property in the numerical computation of
optimal trajectories, splitting the optimization horizon at the turnpike. They proceed as follows: given the turnpike equilibrium $x^{e}$, the optimization horizon $T>0$, an
initial condition $x(0)$ and, a terminal condition $x(T)$, they compute an
optimal trajectory $x_{1}(\cdot )$ with finite horizon $T_{1}<T$ and initial
and terminal conditions $x_{1}(0)=x(0)$, $x_{1}(T_{1})=x^{e}$; and an
optimal trajectory $x_{2}(\cdot )$ with horizon $T_{2}<T-T_{1}$ and initial
and terminal conditions $x_{2}(0)=x^{e}$, $x_{2}(T_{2})=x(T)$. Finally, an
approximation of the optimal trajectory is then obtained by concatenating
the three arcs: $x_{1}(t),$ $t\in \lbrack 0,T_{1}];$ $x^{e},$ $t\in \lbrack
T_{1},T-T_{2}]$ and $x_{2}(t-T+T_{2}),$ $t\in \lbrack T-T_{2},T].$ The
resulting error can be estimated if the speed of convergence towards the
turnpike is known (as in the case of exponential turnpike). They also use a
second approach via Model Predictive Control (MPC) which may have has some advantages. 

To illustrate the turnpike and their methods, they consider a well known harvest example
\cite{Cliff Vincent 1973}, with both bilinear and quadratic objective. Remarkably enough,
the authors do not seem to notice that in the free-endpoint case, the leaving arc ends
always at the same value of $x(T)$, regardless of $x(0)$ and $T$. Something similar
happens in \cite{Grune 2021}: the author, who studies two examples of optimal investment
problems, states literally: ``\textit{Without any terminal constraints all predictions end
  in $x=2$}'', but does not delve into this happening. We shall see that this is a general
property of turnpikes with free-endpoint solutions.

As a matter of fact, one of us had already noticed this in the previous paper \cite{Bayon
  2017}. There, a model of renewable resource exploitation in an open-access fishery
\cite{Agnarsson 2008}, more detailed and general than \cite{Cliff Vincent 1973}, is
studied. It was noticed that, without constraints on the terminal state (which force the
trajectory to leave the turnpike), the solution spontaneously leaves the turnpike in order
to reduce the cost of the overall trajectory, and the leaving arc always ends at the same
value of $x(T)$, for all $T$.

In this note, we intend to settle the question of the entry and leaving arcs of the
turnpike in the generic hyperbolic situation for \emph{variational problems}. The key
point was suggested in \cite{Bayon 2017} but not led to its natural consequence there. In
short, and loosely speaking, our statement can be summarized as follows (for autonomous
problems in $\mathbb{R}$): assume $P=(x_P,\dot{x}_P)$ is a turnpike for a problem with
initial condition $x_0$ and free terminal condition, and let $T(x,\dot{x})=0$ denote the
equation giving the transversality condition. Then:

\textbf{Statement:} There is a function $C(x,\dot{x})$ such that:
\begin{itemize}
\item The entry arc of the turnpike starts at
$$
Q_{e}=\left\{ x=x_0\right\} \cap \left\{C(x,\dot{x})=C(x_P,\dot{x}_{P}) \right\}.
$$
\item The leaving arc of the turnpike ends at 
  $$
  Q_{l}=\left\{ T(x,\dot{x})=0\right\}\cap
  \left\{C(x,\dot{x})=C(x_P, \dot{x}_P) \right\}.
  $$
\end{itemize}

The function $C(x,\dot{x})$ is well-known to any practitioner: it is the function whose
level sets are the extremals \cite{kielhofer 2018}. Certainly, the statement needs to be properly
formalized, but its spirit should be clear to anyone familiar with the turnpike
property. It is also more general (the problem may have both endpoints fixed, or none).

The main tool in our argument is to study the phase space of the plane vector field equivalent to the Euler equation associated to the variational problem. This vector field has very nice properties (among other things, its trajectories are both the extremals of the problem and the level sets of $C(x,\dot{x})$) and a direct application of the classical results on ordinary differential equations is enough to prove the statement.

The consequences of that result are straightforward: in order to determine the entry and leaving arcs, one only needs to know the intersection points between $C(x,\dot{x})$, the transversality condition and/or the initial and final conditions (if any). Once these points are known, the entry arc can be computed by {}forward{} integration, and the leaving arc by backwards integration, {}as the question has become an initial value problem at this point{}.

We hope this work provides a useful support for the study of long-term autonomous variational (and possibly control) problems near a steady state.

{}Our results are all straightforward consequences of the standard results on continuous dependence on parameters of solutions of ordinary differential equations, as well as the local structure of hyperbolic singularities. Despite this fact, we dedicate Section ~\ref{sec:saddles} to a thorough description of the geometric setting of the problem, with the aim of helping the reader understand the situation. We hope this is clearer, briefer and simpler than a technical proof which would provide no insight and would be no more informative than what we provide.

After the formal statements in Section \ref{sec:results} and a suggestion for an approximate algorithm, we dedicate Section \ref{sec:example} to a hopefully illustrative example, Section \ref{sec:simulations} to numerical computations in it. A final section provides some remarks on the $n$-dimensional case.
{}

\section{Statement of the problem}
Consider the autonomous variational problem in one dimension:
\begin{equation}
  \label{eq:P1}
    \mathcal{P}\equiv
    \left\{
    \begin{array}{l}
    \displaystyle\min \int_0^T F(x(t), \dot{x}(t))\,dt\\[1em]
    x(0) = x_0
    \end{array}
    \right.
\end{equation}
where $F$ is a $\mathcal{C}^2(\mathbb{R}^{2})$ function, and $T$ is large enough. It is well known since Samuelson \cite{Samuelson 1958} that many of
these problems have a \emph{turnpike}: a value $x_{P}$ such that ``most'' solutions of
\eqref{eq:P1} pass near it during a long period (i.s. $x(t)\simeq x_{P}$ and
$\dot{x}(t)\simeq 0$ for a ``large'' inner subinterval of $[0,T]$), for
$T\rightarrow \infty$. Moreover, as Zaslavsky has proved \cite{Zaslavski 2008}, there
are also initial and final curves (the \emph{entry} and \emph{leaving} arcs) $\gamma_e$
and $\gamma_l$ such that, as $T\rightarrow \infty$, any solution of that problem is very
near $\gamma_e$ at the beginning, then near $x_{P}$, and finally, it is near $\gamma_l$ in
the end. Of course, all the terms between quotation marks can be properly defined
\cite{Zaslavski 2014}.

However, despite all the results around turnpikes, and as we have remarked in the
Introduction, there is still no programmatic way to find their entry and exit
arcs. The aim of this work is to explicitly show which curves these arcs are and how to
compute them {}in the generic case{}.

\section{Extremal curves, level sets and hyperbolic saddles}\label{sec:saddles}
Given problem \eqref{eq:P1}, Euler's equation
\begin{equation}
  \label{eq:euler}
  \frac{\partial F(x(t),\dot{x}(t))}{\partial x} -
  \frac{d}{dt}\left(
    \frac{\partial F(x(t),\dot{x}(t))}{\partial \dot{x}}
    \right)=0
\end{equation}
is best rewritten, after simplifying a common factor $u$, for our purposes, as a vector field, using $x$ and $u$ as subindices to indicate partial differentiation with respect to the first and second variables:
\begin{equation}
  \label{eq:euler-campo}
  \mathcal{E}\equiv\left\{
    \begin{array}{l}
      \dot{x} = u\\[3pt]
      \displaystyle
      \dot{u} = \frac{F_x - u F_{x u}}{F_{uu}}
    \end{array}
  \right.
\end{equation}
This vector field might have singularities where $F_{uu}=0$ (this implies, for instance, that if the problem is strictly convex in $u$, then there are no such singularities). The extremal curves (solutions to Euler's equation) then coincide with the trajectories of $\mathcal{E}$. Moreover, as the problem is autonomous, it is well known (see, for instance \cite{kielhofer 2018}) that the following function

\begin{equation}\label{eq:constant-on-extremals}
  C(x,u) = F(x,u) - u F_u(x,u)
\end{equation}
is constant in the extremals. Thus, \emph{extremals, as $1$-dimensional manifolds, are the level sets of $C(x,u)$} in $\mathbb{R}^{2}$, for the problem $\mathcal{P}$.

Let us work away from the points where $F_{uu}=0$, that is, we limit ourselves to an open set $W$ where $F_{uu}(x,u)\neq 0$. Consider an equilibrium point $P$ of $\mathcal{E}$, that is: a point with $u=0, F_{x}-uF_{xu}=0$. The linear part of $\mathcal{E}$ is always of the form
\begin{equation*}
L=  \begin{pmatrix}
    0 & \star_{1}\\
    1 & \star_2
  \end{pmatrix}
\end{equation*}
where the stars are unknown values. Except in degenerate cases, $P$ is then either a
center-focus (when both eigenvalues of $L$ are complex), a node (both eigenvalues of $L$
are real and have the same sign) or a hyperbolic saddle (real eigenvalues with different
sign). Obviously, the nature of either center-foci or nodes prevents such a point from
being a turnpike with entry and leaving arcs: if $P$ is a center, trajectories turn around it, if it is a focus, then
they either converge to it (so that $P$ is not strictly speaking a turnpike) or move away
from it (again, not a turnpike). For the same reasons as for foci, nodes cannot be
turnpikes. Hence, turnpikes with entry and leaving arcs correspond, in the non-degenerate case, to hyperbolic saddles, as is well known.

Assume then that $P=(x_P,u_P)$ is a hyperbolic saddle of $\mathcal{E}$, which by
definition will have $u_P=0$ (this is exactly what makes $P$ a turnpike: near $P$, the
velocity of $\mathcal{E}$ tends to $0$ and extremals spend ``a long time'' near $P$). On
the other hand, we have $F_x(x_P,u_P)-u_PF_{xu}(x_p,u_P)=0$, which becomes at $P$ just
$F_x(x_P,0)=0$. As $P$ is a hyperbolic saddle, there are two invariant manifolds adherent
to $P$: the stable $X_{s}$ and unstable $X_{u}$ ones, meeting transversely at $P$ (see
Figure \ref{fig:hyperbolic-saddle}: the stable manifold ``falls'' towards $P$ and the
unstable one ``goes away'' from it). As these manifolds are unions of extremal curves
(they are trajectories of $\mathcal{E}$), they  correspond also to level curves of $C(x,u)$ and, as $P$ belongs
to both, if we denote by $M=X_s\cup X_{u}$ their union, must necessarily have:
\begin{equation*}
  M \equiv C(x,u) = C(x_P,0).
\end{equation*}
That is, the invariant set near $P$ is given by $C(x,u)=C(x_P,0)$.

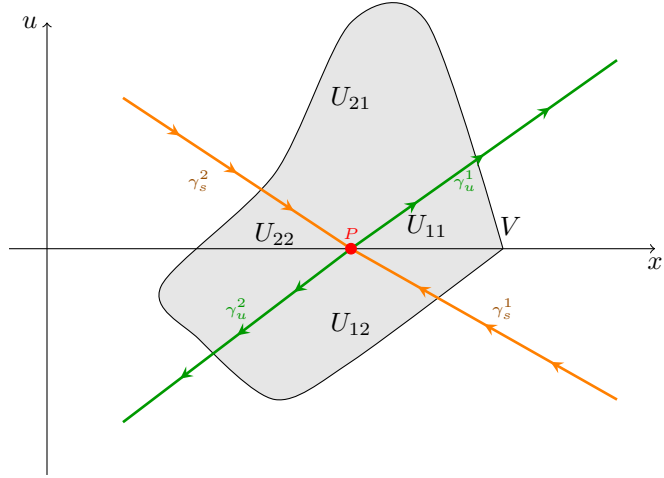
\begin{figure}[h!]
  \centering
  \begin{tikzpicture}[decoration={markings, 
    mark= at position 0.25 with {\arrow{stealth}},
    mark= at position 0.5 with {\arrow{stealth}},
    mark= at position 0.75 with {\arrow{stealth}}}]
  \draw[fill=gray!20!white] plot[smooth] coordinates {(6,0)   (5,3)   (4,3)   (3,1)
    (1.5,-0.5)    (2,-1.2)   (3, -2)   (4, -1.5)  (6,0)};
  \draw(6.1,0.3) node{$V$};
  \draw(5,0.3) node{$U_{11}$};
  \draw(4,2) node{$U_{21}$};
  \draw(3,0.2) node{$U_{22}$};
  \draw(4,-1) node{$U_{12}$};
  \draw[->] (-0.5,0) -- (8,0) node[anchor=north]{$x$};
  \draw[->] (0,-3) -- (0,3) node[anchor=east]{$u$};
  \draw[color=orange, line width=1pt,smooth,postaction={decorate}] (7.5,-2) -- (4,0);
  \draw[color=orange, line width=1pt,smooth,postaction={decorate}] (1,2)  -- (4,0);
  \draw[color=green!60!black, line width=1pt,postaction={decorate}] (4,0) --(7.5,2.5); 
  \draw[color=green!60!black, line width=1pt,postaction={decorate}] (4,0) -- (1,-2.3);
  \draw[fill=red,red](4,0) circle(2pt) node[anchor=south]{\tiny{$P$}};
  \draw[orange!60!black] (6,-0.8) node{\tiny{$\gamma_s^1$}};
    \draw[orange!60!black] (2,0.9) node{\tiny{$\gamma_s^2$}};
    \draw[green!60!black] (2.5,-0.8) node{\tiny{$\gamma_u^{2}$}};
    \draw[green!60!black] (5.5,0.9) node{\tiny{$\gamma_u^{1}$}};
\end{tikzpicture}
\caption{Hyperbolic saddle $P$ and the open sets $U_{ij}$.}
  \label{fig:hyperbolic-saddle}
\end{figure}

Near $P$, the set $M$ can be divided into $4$ different trajectories of $\mathcal{E}$: $\gamma_s^1, \gamma_s^2$, which are the two components of $X_s\setminus\left\{ P \right\}$ and $\gamma_{u}^1$, $\gamma_u^2$ for $X_{u}\setminus \left\{ P \right\}$. Any connected open set $V$ containing $P$ with sufficiently smooth border is divided by these four curves into four open subsets: $U_{11}$, $U_{12}$, $U_{21}$, $U_{22}$, each $U_{ij}$ corresponding to the ``angle'' defined by $\gamma_s^i$ and $\gamma_u^j$, in that order (See Figure \ref{fig:hyperbolic-saddle}).

The solutions of the variational problem $\mathcal{P}$ are extremals (so, they correspond
to trajectories of $\mathcal{E}$) which verify the initial condition $x(0)=x_0$ and also
satisfy the transversality condition $F_u(x(T),u(T))=0$. The equation given by the
trasnversality condition $F_{u}(x,u)=0$ defines (usually) a curve in the
$(x,u)$-plane. Figure \ref{fig:base-case} shows the ``general'' situation in which we find
ourselves. The trajectory {}$\gamma_{e}$, part of the stable manifold, and $\gamma_{l}$, part of the unstable one,{} are
the entry and leaving arcs, respectively.

\begin{figure}[h!]
  \centering
  \begin{tikzpicture}[decoration={markings, 
    mark= at position 0.25 with {\arrow{stealth}},
    mark= at position 0.5 with {\arrow{stealth}},
      mark= at position 0.75 with {\arrow{stealth}}}] 
  \draw[->] (-0.5,0) -- (8,0) node[anchor=north]{$x$};
  \draw[->] (0,-3) -- (0,3) node[anchor=east]{$u$};
  \draw[white,fill=cyan!30!white] plot[smooth]
  coordinates { (7.5,-2.3)   (5,-1) (4,-0.5)  (3,-1.1)  (1.5, -2.3) (1,-2.7) }
  -- (1.4,-2.35) -- (1.15,-2.2) -- (4,0) -- (7.5,-2) -- cycle;
    \path[fill=yellow,line width=0pt] (4,-.025) -- (4.5,-0.32) -- (3.6,-0.32) -- (4,-0.025);
  \draw[color=black,dashed] (7.5,3) -- (7.5,-3) node[anchor=west]{$x=x_0$};
  \draw[black,postaction={decorate}, line width=0.5pt] plot[smooth]
  coordinates { (7.5,-2.3)   (5,-1) (4,-0.5)  (3,-1.1)  (1.5, -2.3) (1,-2.7) };
  \draw (4.2,-0.8) node{\tiny{$\gamma$}};
  \draw[color=orange, line width=1pt,smooth,postaction={decorate}] (7.5,-2) -- (4,0);
  \draw[color=orange, line width=1pt,smooth,postaction={decorate}] (1,2)  -- (4,0);
  \draw[color=green!80!black, line width=1pt,postaction={decorate}] (4,0) --(7.5,2.5); 
  \draw[color=green!80!black, line width=1pt,postaction={decorate}] (4,0) -- (1,-2.3);
  \draw[fill=red,red](4,0) circle(2pt) node[anchor=south]{\tiny{$P$}};
  \draw[orange,dashed,line width=1pt] plot[smooth] coordinates {(-0.3,0.3) (0,-0.3) (1, -2) (2,-2.5) (4, -2.5) (5,-3)};
  \draw[orange!80!black] (0,-0.3) node[anchor=east]{\tiny{$F_u(x,u)=0$}};
  ;
  \draw[dashed,line width=0.5pt,fill=black] (1.15,-2.17) circle(.5pt) -- (1.15,0) circle(1pt);
  \draw (1.3,0) node[anchor=south east]{\tiny{$x_{l}$}};
  \draw[dashed,line width=0.5pt,fill=black] (1.45,-2.35) circle(.5pt) -- (1.45,0) circle(1pt);
  \draw (1.7,0) node[anchor=south]{\tiny{$x(T)$}};
  \draw (1.3,0.03)node[anchor=south] {\tiny{$\leftarrow$}};
  \draw[orange!80!black] (6,-0.9) node{\tiny{$\gamma_{e}$}};
  \draw[green!80!black] (2.5,-0.9) node{\tiny{$\gamma_{l}$}};
  \draw[fill=black](1.15,-2.17) circle(1.5pt);
  \draw(0.85,-2.17) node{$Q_{l}$};
  \draw(7.9,-2) node{$Q_{e}$};
  \draw[fill=black](7.5,-2) circle(1.5pt);
\end{tikzpicture}
\caption{Hyperbolic saddle $P$ (turnpike), extremal ($\gamma$), and entry ($\gamma_{e}$) and leaving ($\gamma_{l}$) arcs. In yellow, the ``slow'' zone. As long as there are no singularities of $\mathcal{E}$ in the cyan zone, the turnpike property holds inside it, and as  $T\rightarrow \infty$, the corresponding extremal of $\mathcal{P}$ approaches $\gamma_{e}$ at the beginning and $\gamma_{l}$ at the end. The entry arc starts at $Q_{e}$ and the leaving arc ends at $Q_{l}$.}
  \label{fig:base-case}
\end{figure}
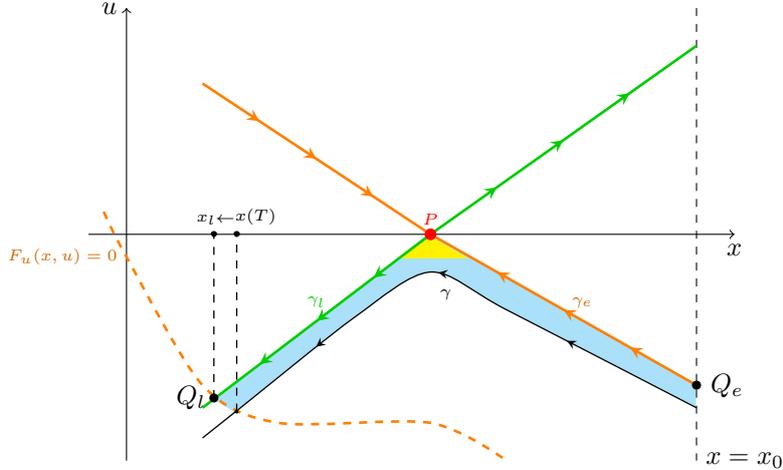

The intersection points between the transversality condition $F_u(x,u)=0$ and $M$ are, key
in our statements. These points are the solutions of the system of equations:
\begin{equation*}
  \left\{
    \begin{array}{l}
      F(x,u) -u F_u(x,u)= C(x_P,0)\\
      F_u(x,u) = 0
    \end{array}
  \right.
\end{equation*}
which, after simplifying, becomes (see \cite{Bayon 2017}, where this system of equations appeared for the first time):
\begin{equation}
  \label{eq:bayon}
  \left\{
    \begin{array}{l}
    F(x,u) = C(x_P,0)\\
    F_u(x,u) = 0
    \end{array}
  \right.
\end{equation}
In the problem $\mathcal{P}$, the initial value $x(0)=x_0$ is set. Assume
$Q_{e}=(x_0,u_{e})$ belongs to $x=x_{0}\cap \{C(x,u)=C(x_P,0)\}$ and to the stable
manifold of $P$, and let $Q=(x_{l},u_{l})$ be the solution of \eqref{eq:bayon} in the
unstable manifold (as in Figure \ref{fig:base-case}) which is nearest to $P$. Under these assumptions, the Turnpike property happens
relative to $P$ (as in Figure \ref{fig:base-case}) and extremals with $x(0)=x_0$ start
near $Q_{e}=(x_{0},u_{e})$, so that $u(0)\rightarrow u_{e}$, and finish near
$Q_{l}$, so that and $x(T)\rightarrow x_{l}$ as
$T\rightarrow\infty$.

If there existed a solution $R=(x_{e}, u_{e})$ of \eqref{eq:bayon}, belonging to
the stable manifold and satisfying the transversality condition (this case is not plotted
in Figure \ref{fig:base-case}), a dual argument using $f(x,-\dot{x})$ shows that there are
extremals of the variational problem with no initial or terminal condition:
\begin{equation}
  \label{eq:P1-libre}
    \mathcal{P}^{\prime}\equiv
    \min \int_0^T F(x(t), \dot{x}(t))\,dt
\end{equation}
with starting points $x(0)\rightarrow x_{e}$ (and endpoints ending at $Q_{l}$ as above). 

This theoretical description which is just a qualitative expression of the well-known results on the continuous dependence of solutions of ODEs on the parameters, and of the local structure of hyperbolic singularities (see \cite{Ilyashenko}, for example) is enough to prove our results, so that instead of proofs, we just refer to this section.

Our statements have two versions: one in which a solution $\gamma$ of $\mathcal{P}$ is
already known, and one in which all depends just on the solutions of \eqref{eq:bayon}.

\section{Statements of the results}\label{sec:results}

As explained in the introduction, $F(x,u)$ is of class $\mathcal{C}^2$ in $\mathbb{R}^2$, and all our statements are in an open set $W\subset \mathbb{R}^{2}$ where $F_{uu}(x,u)\neq 0$. We shall make frequent reference to the vector field $\mathcal{E}$ defined in (\ref{eq:euler-campo}).

Our first result assumes the existence of an extremal ``sufficiently'' near a hyperbolic turnpike:
\begin{theorem}\label{the:structure-with-extremal}
  Let $P=(x_P,u_P)$ be a hyperbolic saddle of $\mathcal{E}$ and $\gamma$ an extremal of $\mathcal{P}$ included in an open set $U\subset W$ containing $P$ which admits a subdivision $U_{ij}$ for $i,j\in \left\{ 1,2 \right\}$ as above. We assume the following conditions:
  \begin{enumerate}
  \item\label{enum:transverse-trans} The curve $F_u(x,u)=0$ meets $\gamma_u^1$ and
    $\gamma$ transversely at the points $Q=(x_{l},u_{l})$ and $(x(T),u(T))$.
  \item\label{enum:transversality-parametrized} That curve $F_u(x,u)=0$ admits an
    injective parametrization near $Q$, $\eta:[-1,1]\rightarrow \mathbb{R}^2$ with
    $\eta(0)=Q$, $\eta(1)=(x(T),u(T))$ such that $\eta$ is transverse to any extremal
    meeting it.
  \item\label{enum:transverse-gamma} The extremals $\gamma_s^1$ and $\gamma$ meet the
    manifold $x=x_0$ transversely at $(x_0,u_{e})$, $(x_0,u_0)$ respectively.
  \item\label{enum:no-more-sings} The open set $V$ (delimited by $\gamma^1_s$ and
    $\gamma_u^1$, $\eta$ and the line $x=x_0$ contains no more singularities of
    $\mathcal{E}$.
  \end{enumerate}
  Then:
  \begin{enumerate}
  \item For any $\overline{T}>T$, the problem
    \begin{equation*}
      \mathcal{S}_{\overline{T}} \
          \left\{
    \begin{array}{l}
    \displaystyle\min \int_0^{\overline{T}} F(x(t), \dot{x}(t))\,dt\\[1em]
    x(0) = x_0
    \end{array}
    \right.
  \end{equation*}
  has an extremal $\gamma$ which is totally included in $V$.
\item For any $\epsilon >0$ there is $T_{\epsilon}>T$, $T_{\epsilon,{e}}>0$ and $T_{\epsilon,\infty}<T$, such that, for $\overline{T}>T_{\epsilon}$, the corresponding extremal  $\gamma_{\overline{T}}$ satisfies:
  \begin{itemize}
  \item The metric distance between $\gamma_{\overline{T}}:[0,T_{e}]\rightarrow \mathbb{R}^2$ and $\gamma_{e}$ is less than $\epsilon$.
  \item The metric distance between $\gamma_{\overline{T}}:[T_{e}, T_{l}]\rightarrow \mathbb{R}^2$ and $P$ is less than $\epsilon$.
  \item The metric distance between $\gamma_{\overline{T}}:[T_{l}, \overline{T}]\rightarrow \mathbb{R}^2$ and $\gamma_{l}$ is less than $\epsilon$.
  \end{itemize}
  Moreover, one can also choose $T_{e}$ and $T_{l}$ such that $T_{l}-T_{e}\rightarrow \infty$ as $\epsilon \rightarrow 0$ and $T_l,T_e<K$ for some $K<\infty$.
  \end{enumerate}
\end{theorem}
\begin{proof}
  The first conclusion follows from the continuous dependence of solutions of an ordinary differential equation on the parameters (and from all the qualitative descriptions in Section \ref{sec:saddles}). The second one follows also from the local structure of hyperbolic saddle singularities, see for instance \cite{Ilyashenko}.
\end{proof}

\begin{definition}\label{def:entry-exit}
The curve $\gamma_{e}$ from $x=x_0\cap \gamma_{e}$ to $P$ is called the \emph{entry arc} to the turnpike $P$. The curve $\gamma_{l}$ from $P$ to $Q$ is called the \emph{leaving arc} of the turnpike $P$.
\end{definition}

The previous statement seems to require a lot from the equation. As it happens, most of
the hypotheses are just technical and will hold in generic situations. On the other hand,
we can also say a lot (locally) if we just know that the transversality condition meets
$\gamma_{l}$ transversely. The following result is again a straightforward consequence of the local structure of hyperbolic saddles and the continuous dependence on the parameters of solutions of ODEs. All the statements are inside $W\subset F_{uu}(x,u)\neq 0$.
\begin{theorem}\label{the:TC-meets-gamma-omega}
  Let $P$ be a hyperbolic saddle and $\gamma_{l}\subset \gamma_u^1\cup \gamma_u^2$
  one of the components of the unstable manifold. Assume that the transversality condition
  $Tr\equiv F_u(x,u)=0$ meets $\gamma_{l}$ transversely at
  $Q_{l}=(x_{l}, u_{l})$ and that there are no more intersection points between $Q_l$
  and $P$ belonging to $\gamma_{l}$. Then there is $\epsilon>0$ and a parametrization
  $\tau:[-\epsilon,\epsilon]\rightarrow T_c$ with $\tau(0)=Q_{l}$ such that the two
  extremals containing $\tau(-\epsilon)$ and $\tau(\epsilon)$ satisfy all the properties
  of Theorem \ref{the:structure-with-extremal}. As a consequence, $P$ is a turnpike for
  $\mathcal{P}$. Moreover, assume $\gamma_s^1$ is to the right of $P$ and $\gamma_s^2$ is to its
  left. Assume, for simplicity, that $Q_{l}$ is to the left of $P$
  (i.e. $x_{l}<x_P$). Then:
  \begin{enumerate}
  \item The point $P$ is a turnpike for $\mathcal{P}$ for $x_0\in (x_P,x_P+\epsilon)$, for
    some $\epsilon>0$, $\gamma_{l}$ (from $P$ to $Q$) is the leaving arc and
    $\gamma_s^1$ (from $x=x_0$ to $P$) is the entry arc.
  \item At the same time, $P$ is a turnpike for $\mathcal{P}$ for
    $x_0\in(x_P-\epsilon,x_P)$ for some $\epsilon>0$, $\gamma_{l}$ (from $P$ to $Q$)
    is the leaving arc and $\gamma_s^2$ (from $x=x_0$ to $P$) is the entry arc.
  \end{enumerate}
\end{theorem}

Notice how there is a switch of entry arcs when the initial condition $x_0$ changes from being ``greater than $x_{P}$''  to ``less than $x_P$''. This is easy to see in Figure \ref{fig:base-case}: if $x_0$ is less than $x_P$, the extremals ending near $Q$ must approach, at their beginning, the top-left separatrix for $T\rightarrow \infty$, instead of $\gamma_{e}$. Obviously, for the dual problem (final condition set but initial condition free), it is the leaving arc that changes.

Consider the problem \eqref{eq:P1-libre} with no initial or final condition. Near a hyperbolic singularity $P$ of $\mathcal{E}$ we \emph{may} have a turnpike result if the system of equations \eqref{eq:bayon} has two solutions near $P$. Again, everything is restricted to some open set $W\subset F_{uu}(x,u)\neq 0$.

\begin{theorem}\label{the:free-points}
  With the notations above, assume $P$ is a hyperbolic singularity of $\mathcal{E}$. If $Q_{e}\in \gamma_s^{1}$ and $Q_{l}\in \gamma_u^1$ are two solutions of \eqref{eq:bayon} and there are no more solutions of \eqref{eq:bayon} between $Q_{e}$ and $P$ and $P$ and $Q_{l}$, then $P$ is a turnpike for the problem \eqref{eq:P1-libre}. That is, for $T\rightarrow \infty$, there are extremals $\gamma=(x(t),u(t))$ of \eqref{eq:P1-libre} satisfying:
  \begin{enumerate}
  \item The origin tends to $Q_{e}$: $(x(0),u(0))\rightarrow Q_{e}$,
  \item The end tends to $Q_{l}$: $(x(T),u(T))\rightarrow Q_{l}$,
  \item The curve $\gamma$ approaches the part of $\gamma_s^{1}$ between $Q_{e}$ and $P$ at the beginning (the entry arc),
  \item The curve $\gamma$ approaches the part of $\gamma_u^{1}$ between $P$ and $Q_{l}$ at the end (the leaving arc).
  \end{enumerate}
\end{theorem}
\begin{proof}
  As previously, the proof is a straightforward application of the description in \ref{sec:saddles}, the structure of hyperbolic singularities and the continuous dependence on parameters of solutions of ODEs.
\end{proof}

Finally, consider the problem with fixed endpoints:
\begin{equation}
  \label{eq:P1-fixed-endpoints}
      \overline{\mathcal{P}}\equiv
    \left\{
    \begin{array}{l}
    \displaystyle\min \int_0^T F(x(t), \dot{x}(t))\,dt\\[1em]
    x(0) = x_0,\ x(T) = x_T
    \end{array}
    \right.
\end{equation}
In this case the statements holds regardless of the transversality condition.

\begin{theorem}\label{the:fixed-endpoints}
  Assume $P$ is a hyperbolic singularity of $\mathcal{E}$. Assume $x=x_0$ meets $\gamma_1^s$ at $Q_{e}$ and $x=x_T$ meets $\gamma_1^u$ at $Q_{l}$. Then $P$ is a turnpike for the problem \eqref{eq:P1-fixed-endpoints}, and there is an open neighborhood $V$ of $P$ such that, for $T\rightarrow\infty$ any extremal $\gamma$ of \eqref{eq:P1-fixed-endpoints} included in $V$ satisfies  statements (1)--(4) of Theorem \ref{the:free-points}.
\end{theorem}

Notice that $Q_{e}$ and $Q_{l}$ in previous results can be easily computed using
the fact that $\gamma_i^{s}$ and $\gamma_i^u$ are level sets of $C(x,u)$. Thus, if they
exist, then
\begin{equation*}
  Q_{e} \in \left\{ C(x,u)=C(P) \land x=x_0 \right\},\,\,\,
  Q_{l} \in \left\{ C(x,u)=C(P) \land x=x_T \right\}
\end{equation*}
However, those solution sets might have more than one point, and one needs to verify which (if any) can be a candidate.

{}
\subsection{Suggestion for an approximate algorithm}
From the discussions above, the following method is suggested to use the turnpike and the entry and leaving arcs as approximate extremals for $T\gg 0$. Specifically, for Problem \eqref{eq:P1} with $x(0)$ fixed and $x(T)$ free:
\begin{enumerate}[i)]
\item State a tolerance $\epsilon>0$.
\item Find the possible turnpike $P=(x_P,0)$. This requires studying the phase space of \eqref{eq:euler-campo}, its singularities and the separatrices $\gamma_s^i$ and $\gamma_u^i$, for $i=1,2$. For this one can just use the level set $C(x,u)=C(x_P,0)$.
\item Find the adequate $Q_e$ and $Q_l$. As explained above, $Q_e$ belongs to $x=x_0$ and $C(x,u)=C(x_P,0)$, whereas $Q_l$ is found using system \eqref{eq:bayon}.
\item From $Q_e$ compute the trajectory $\gamma_{e}(t)$ of $\mathcal{E}$ with $\gamma_e(0)=Q_e$ and ending at $|\gamma_e(T_e)-x_P|<\epsilon$. This is an IVP integrated until some condition is met.
\item From $Q_l$ compute the trajectory $\gamma_l(t)$ of $\mathcal{E}$ with $\gamma_l(T_{l})=Q_l$ and $|\gamma_l(0)-x_P|<\epsilon$. This is a backwards IVP integrated until some condition is met.
\end{enumerate}
After those computations, if $T>T_e+T_l$, then any extremal $\gamma_T$ of \eqref{eq:P1} can be approximated by the turnpike as:
\begin{equation}
  \label{eq:approximation}
  \gamma_T(t) \simeq \left\{
    \begin{array}{l}
      \gamma_e(t)\  \mathrm{if}\ t\in[0,T_{e})\\
      x_{P}\ \mathrm{if}\ t\in[T_e,T-T_l]\\
      \gamma_l(t)\ \mathrm{if}\ t\in(T-T_l,T]
    \end{array}
  \right.
\end{equation}
{}

\section{An example: shallow lakes}\label{sec:example}
In this section we showcase the well-known shallow lakes model without discount (see for instance \cite{wagener 2003} for the details), with a modified cost function to prevent, in our example, the issues with the logarithm. The problem to be solved is, initially, the Optimal Control problem with control variable $v$:
\begin{equation}
  \label{eq:fisheries}
    \mathcal{P}\equiv
    \left\{
    \begin{array}{l}
      \displaystyle\max \int_0^T \left(v^{2} - c x^{2}\right)\,dt\\[0.5em]
      \dot{x} = v - b x + r \displaystyle\frac{x^{2}}{x^2+1}\\
    x(0) = x_0
    \end{array}
    \right.
  \end{equation}
  with $c, b, r$ positive constants. This is, in fact, a variational problem, as $v$ can be expressed as a function of $x, \dot{x}$ and there are no restrictions. Thus, we shall in fact study the variational problem
  \begin{equation}
    \label{eq:fisheries-var}
    \mathcal{P}\equiv
    \left\{
    \begin{array}{l}
      \displaystyle\max \int_0^T
      F(x, \dot{x})
      \,dt\\[0.8em]
    x(0) = x_0
    \end{array}
    \right.    
  \end{equation}
  with
  \begin{equation*}
    F(x,\dot{x}) =
b^2 x^2-\frac{2 b r x^3}{x^2+1}+2 b x \dot{x}-c x^2+\frac{r^2 x^4}{\left(x^2+1\right)^2}-\frac{2 r x^2 \dot{x}}{x^2+1}+\dot{x}^2
  \end{equation*}
  and we shall set the value of the constants to $r=1, c=0.1$, and $b=0.7$. The Euler equation for this problem is, once divided by $\dot{x}$:
  \begin{multline}
    \label{eq:fisheries-euler}
    \frac{1}{(1+x^2)^3}\bigg(
      x^6 (-2 \ddot{x}-1.4)+x^4 (-6 \ddot{x}-5.6)\bigg.+
    x^2 (-6 \ddot{x}-4.2)-\\
    \bigg.(2 \ddot{x}+0.78 x^7+2.34 x^5+6.34 x^3+0.78 x
    \bigg) = 0
  \end{multline}
  And the vector field associated to this second order equation is, in the $(x,u)$ plane corresponding to $(x,\dot{x})$:
  \begin{equation}
    \label{eq:fisheries-vector}
    \mathcal{E} \equiv
    \left\{
      \begin{array}{l}
        \dot{x} = u\\
        \dot{u} = \displaystyle
        \frac{0.39 x \left(x^6-1.79487 x^5+3. x^4-7.17949 x^3+8.12821 x^2-5.38462 x+1\right)}{\left(x^2+1\right)^3}
      \end{array}
    \right.
  \end{equation}
  The denominator in $\dot{u}$ is never $0$, so that $\mathcal{E}$ is well-defined in all $\mathbb{R}^{2}$. The vector field $\mathcal{E}$ has three singular points: two hyperbolic saddles $P_1=(0,0)$ and $P_2\simeq(1.5062,0)$ and a center/focus, $O=(0.2747,0)$. Figure \ref{fig:fisheries-vector} shows the structure of $\mathcal{E}$ near its singularities (in red).

  The function whose level sets are the extremals (the trajectories of $\mathcal{E}$) is
  \begin{equation*}
    C(x,u) = x^2 \left(\frac{x^2}{\left(x^2+1\right)^2}-\frac{1.4 x}{x^2+1}+0.39\right)-u^2
  \end{equation*}
  so that we need to focus our attention on the level sets
  \begin{equation*}
    L_1 \equiv C(x,u) = C(P_1) = 0
  \end{equation*}
  and
  \begin{equation*}
    L_2 \equiv C(x,u) =C(P_2) = -0.097
  \end{equation*}
  Finally, the transversality condition in this case is given by
  \begin{equation*}
    Tr \equiv 2 u + 1.4 x - \frac{2 x^2}{1 + x^2}=0.
\end{equation*}
In Figure \ref{fig:fisheries} we have plotted the sets $L_1$ (cyan), $L_2$ (yellow) and $Tr$ (black). Notice how $Tr \cap L_1$ is just the hyperbolic point $P_1$ whereas $Tr\cap L_2$ has two points, one above $u=0$ and the other one below (both in green).

  \begin{figure}[h!]
    \centering
    \includegraphics[width=8cm]{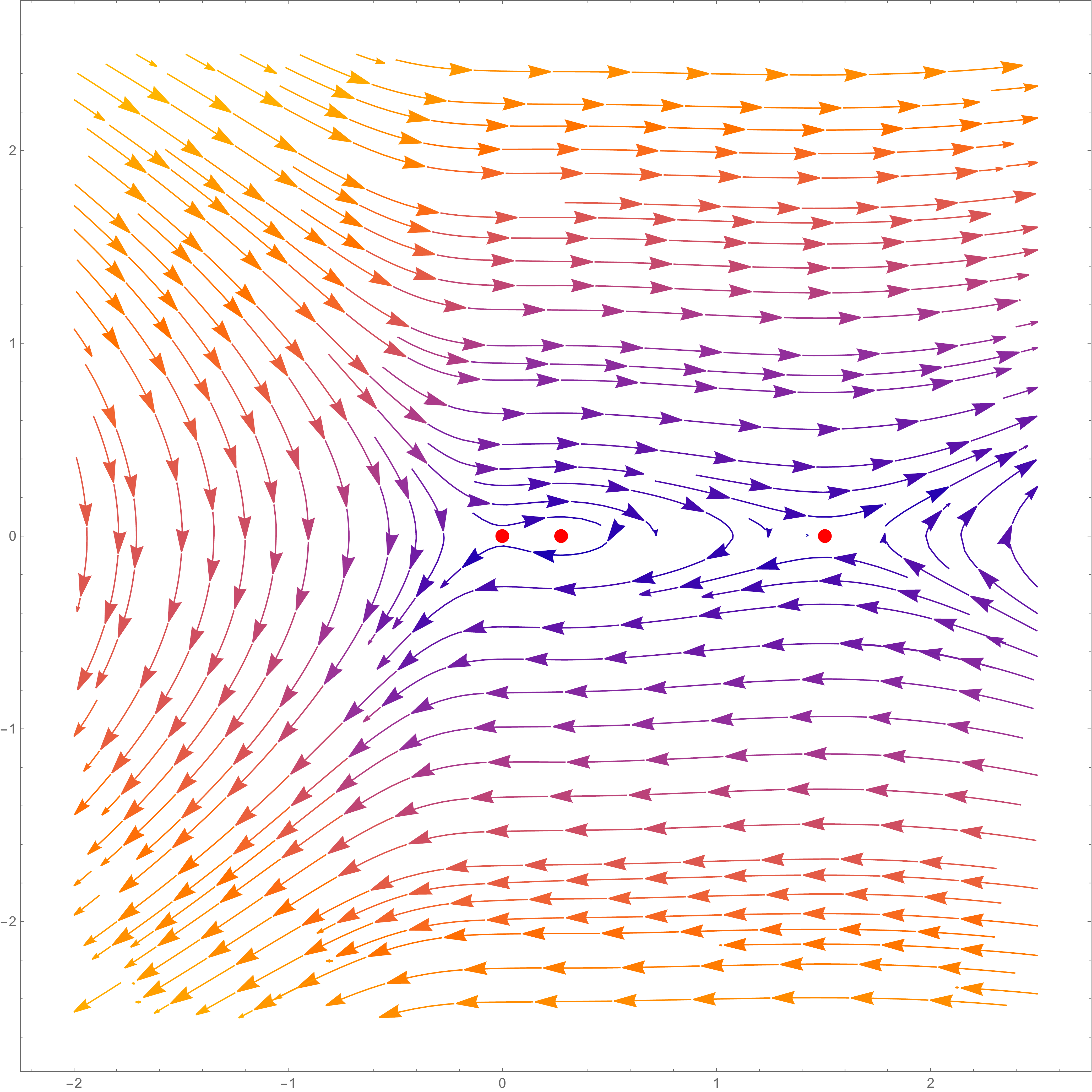}
    \caption{Stream lines of $\mathcal{E}$ in the example. The red dots are its singulaties, at $u=0$, $x\in \left\{ 0,0.2747,1.5062 \right\}$.}
    \label{fig:fisheries-vector}
  \end{figure}
  
  Surprisingly enough, the transversality condition (in black in Figure \ref{fig:fisheries}) only meets the curve $M_{1}\equiv C(x,u)=C(P_1)$ (in cyan) at $(0,0)$ so that our results only apply to $P_1$ in the fixed-endpoints versions {}(because $Tr$ never meets $M_1$ transversely){}.


  Consider the hyperbolic saddle $P_2$. We are going to showcase the four turnpike possibilities for it under problem (\ref{eq:fisheries-var}).

The transversality condition meets the (yellow) curve $M_{2}=C(x,u)=C(P_2)$ at the (green) points $Q_1\simeq(-0.9852,1.1822)$ and $Q_2\simeq(0.9852,-1.971)$. Clearly, the top left and bottom right parts of $M$ are the stable manifolds, call them $\gamma_s^1$ and $\gamma_s^2$, respectively, whereas $\gamma_u^1$ is the bottom-left part. 

\begin{figure}[h!]
  \centering
  \begin{tikzpicture}
    \node (fig) at(0,0) {\includegraphics[width=8cm]{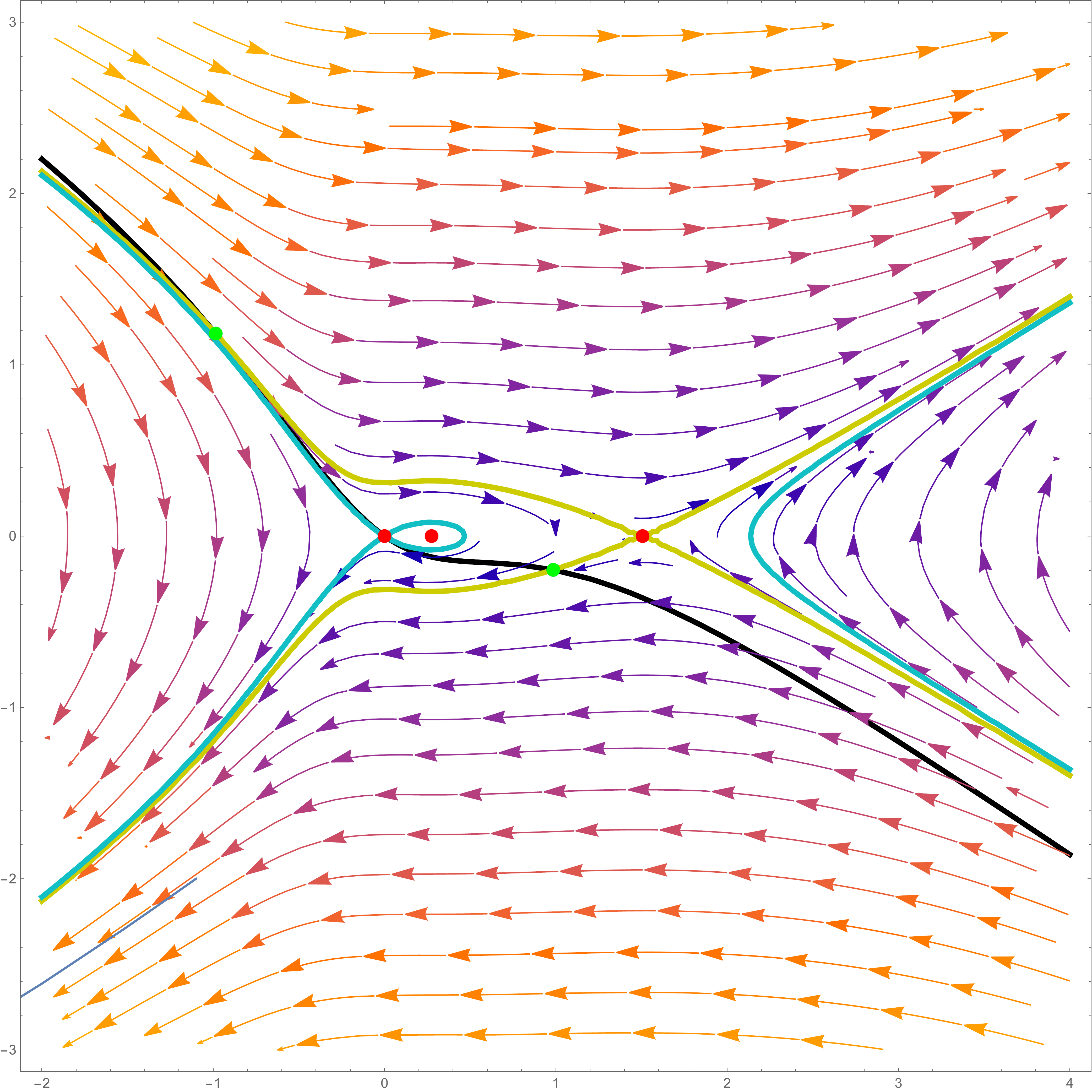}};
    \draw[color=black](0.15,-0.45) node{$Q_2$};
    \draw[color=black](-2.1,1.7) node{$Q_1$};
\end{tikzpicture}
\caption{Hyperbolic structure of the example. The leftmost singularity is hyperbolic but its level set (blue) meets the transversality condition only at the singularity. The level set of the rightmost singularity (yellow) meets the transversality condition twice (at the green dots).}
  \label{fig:fisheries}
\end{figure}

\subsection{Initial condition fixed. Change of entry arc and of turnpike}
Recall that we have called: $\gamma_s^1$ the top-left branch of $M_2$ (in yellow in Figure
\ref{fig:fisheries}) and $\gamma_s^2$ the bottom-right branch (these are the stable
trajectories and will give rise to the entry arcs). Also, $\gamma_u^1$ is the bottom-left
branch, and $\gamma_u^{2}$ the top-right one, which will give rise to the leaving arcs.

In this subsection we are going to study the problem (\ref{eq:fisheries}) (i.e. with
initial condition but no end condition).

If $x_0>1.5062$, the extremals meet the transversality condition near
$Q_2\simeq(0.9852,-1.971)$ for $T\rightarrow \infty$, whatever the value of $x_0$. The
entry arc to the turnpike $P_{2}$ in this case is $\gamma_2^s$ from $x=x_0$ to $P_{2}$:
this happens for any $x_0>1.5062$ because the transversality condition does not meet $M_2$
for $x>1.5026$.

However, the moment $x_0$ is to the left of $P_2$, that is $x_0<1.5062$, the entry arc to turnpike $P$ changes from $\gamma_s^2$ to $\gamma_s^1$ (which is above $u=0$). As $x_{0}\rightarrow -0.9852$ (the $x$-coordinate of $Q_1$), the extremals approach $\gamma_s^1$. The problem with $x_0=-0.9852$ has no solution because $Q_1$ is the only intersection point between an extremal which meets the transversality condition (this is easily seen in the Figure \ref{fig:fisheries}).

Finally, for $x_0<-0.9852$, the candidate extremals for problem $\mathcal{P}_f$ for $T\rightarrow\infty$ approach $P_1=(0,0)$, the intersection point of $Tr$ and $M_1$ (the black and cyan lines in Figure \ref{fig:fisheries}). Thus, there is an entry arc, from $x=x_0\cap Tr$ to $P_1$ but the turnpike is never left, in this case.

\subsection{Initial and final conditions fixed}

When $x(0)=x_0$ and $x(T)=x_T$ are both fixed, the transversality condition plays no role and one needs only study the relation between these conditions and the hyperbolic singularities $P_1$ and $P_2$. For the sake of simplicity, we are only going to show some cases. Let $x_{P_1}=0$ denote the $x$-coordinate of $P_1$ and $x_{P_2}\simeq 1.5062$ the one of $P_2$. Of course, in order to have a turnpike behavior, there must be at least one singularity between $x_0$ and $x_T$.

\begin{itemize}
\item When $x_0>x_{P_2}>x_T$, then $P_2$ is a turnpike and the entry arc is $\gamma_s^2\cap \{x=x_{0}\}$, and the leaving arc is $\gamma_u^1\cap \{x=x_T\}$.
\item On the other hand, if $x_0<x_{P_2}<x_T$ then the situation reverses at $P_{2}$ (we are ``above $u=0$'' and the arcs are now: $\gamma_s^1$ the entry one from $x=x_0$ to $P_2$ and $\gamma_u^2$ from $P_2$ to $x=x_T$.
\item If $x_0,x_T\in (x_{P_1},x_{P_2})$ (that is, both endpoints are between $P_1$ and $P_{2}$) it is easy to realize that $P_1$ is still a turnpike and the entry and leaving paths correspond to $\gamma_s^1$ and $\gamma_u^2$, respectively (starting at $x=x_0$ and ending at $x=x_T$, also).
\item When, say $x_{0}<x_{P_1}$ and $x_T<x_{P_2}$, there are two candidate extremal curves for $T\rightarrow\infty$: one having a turnpike at $P_1$, and the other one at $P_2$; it is necessary here to discern the optimality by other methods (which we shall not do, as this is out of our aim). Obviously, each turnpike has his respective entry and leaving arcs (in this case, $P_1$ arcs are the cyan unbounded curves to its left).
\end{itemize}

\subsection{The free endpoint problem.}
Finally, the free endpoint problem requires the extremals to meet the transversality condition at $x(0)$ and $x(T)$. In our case, only $M_2$ meets $Tr$ twice away from a singularity, whereas $M_1\cap Tr=\{ P_{1} \}$. As far as extremals go, the ``constant curve'' $(x(t),u(t))=P_1$ for all $t\in[0,T]$ is always a candidate trajectory (as it is an extremal which satisfies the transversality conditions). These have obviously constant cost $F(P_1)\times T$.

There is, however, a second possibility giving rise to a true turnpike: the solutions starting near $Q_1$ below $\gamma^1_s$, approaching $P_2$ and ending near $Q_2$ above $\gamma_u^1$. In this case, the entry arc is $\gamma_s^1$ from $Q_1$ to $P$ and the leaving arc is $\gamma_u^1$ from $P_1$ to $Q_2$.

\section{Simulations}\label{sec:simulations}
In this section we plot the simulations corresponding to some of the cases in Section \ref{sec:example}. We have used a budget computer (Intel Core i5 with 16Gb RAM) and Mathematica, with no excessive time used (the simulations can be run in several hours, the longest time taken by the very precise computation of the entry and leaving arcs and, unfortunately, the plotting commands, as the numerical solutions are interpolating functions and their evaluation is quite slow). We restrict ourselves, for the sake of brevity, to the initial problem (\ref{eq:P1}) with $x(0)=x_0=0.5$ and $T\gg 0$.

The above requires computing the turnpike entry arc starting at the point $P_{e}=(0.5,u_e)$, which is the solution of the first equation in \eqref{eq:bayon} with $x=0.5$, that is, $u_e$ is the solution of:
\begin{equation}
  \label{eq:x=0.5-bayon}
  F(0.5, u) = C(P_2) = C(1.5062,0),
\end{equation}
giving $u_e\simeq 0.30751221580$. However, \emph{one needs to compute $u_e$ with a huge precision in order to really obtain a fine approximation to the turnpike.} In our computations, we used $30$ values of precision when computing the solution of \eqref{eq:x=0.5-bayon} (so that $P_2$ was also computed with that precision).

We also need to compute the leaving arc of that turnpike, which requires knowing the point $P_l=Q_2$, solution of (\ref{eq:bayon}):
\begin{equation}
  \label{eq:leaving-bayon}
  \left\{
    \begin{array}{l}
      F(x_{l}, u_{l}) = C(P_2)\\
      Tr(x_l,u_l) = 0
    \end{array}
  \right.
\end{equation}
which gives, as indicated above, $P_l\simeq (0.9852, -1.971)$ (with the same caveat regarding the precision).

Figure \ref{fig:time-63} contains the plot of the extremal $x(t)$ (in blue) and its derivative $u(t)=\dot{x}(t)$ (orange), corresponding to (\ref{eq:P1}) with $x_0=0.5$ and $T=63$. Overlain (in dashed liens) we have plotted the entry arc from $t=0$ to $t=24$, and the leaving arc, from $t=41$ to $t=63$. There is no noticeable difference.

\begin{figure}[h!]
  \centering
  \includegraphics[width=8cm]{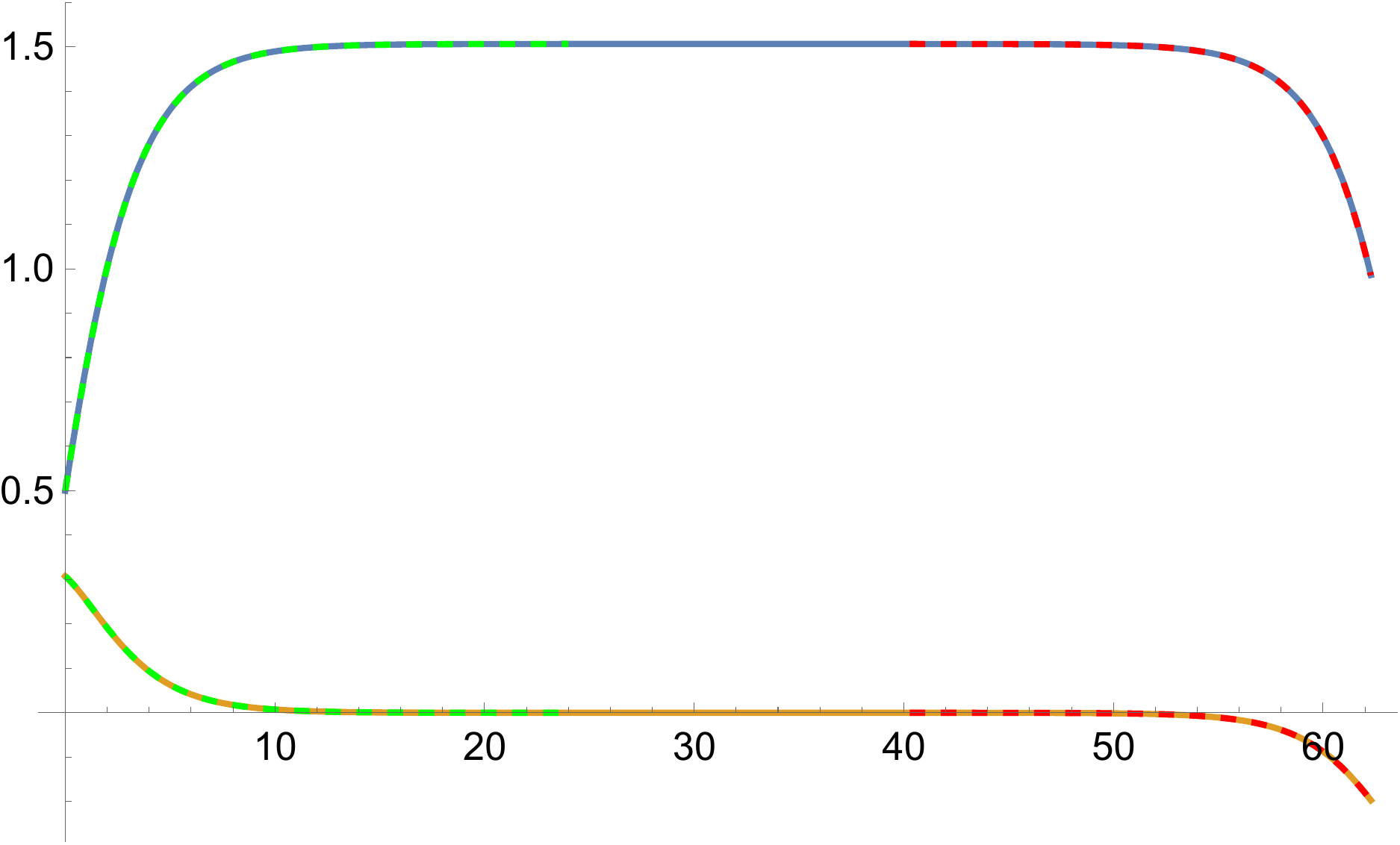}
  \caption{Turnpike entry and leaving arcs compared to solution for $T= 63$.}
  \label{fig:time-63}
\end{figure}

Figure \ref{fig:time-63-diff} shows, on the left, the difference between $x(t)$ and the entry arc for the same $T=63$, and on the right, the difference between $u(t)=\dot{x}(t)$ and the corresponding value on the turnpike, for $t=T-22$ to $T$ (where $22$ is taken as a value where the $x-value$ of the leaving arc is less than $10^{-5}$ from the true turnpike $P_2$). Notice that the time is inverted in the latter plot because we have computed the leaving arc ``backwards''. The errors are, as can be seen, irrelevant to all purposes.

\begin{figure}[h!]
  \centering
  \includegraphics[width=6.2cm]{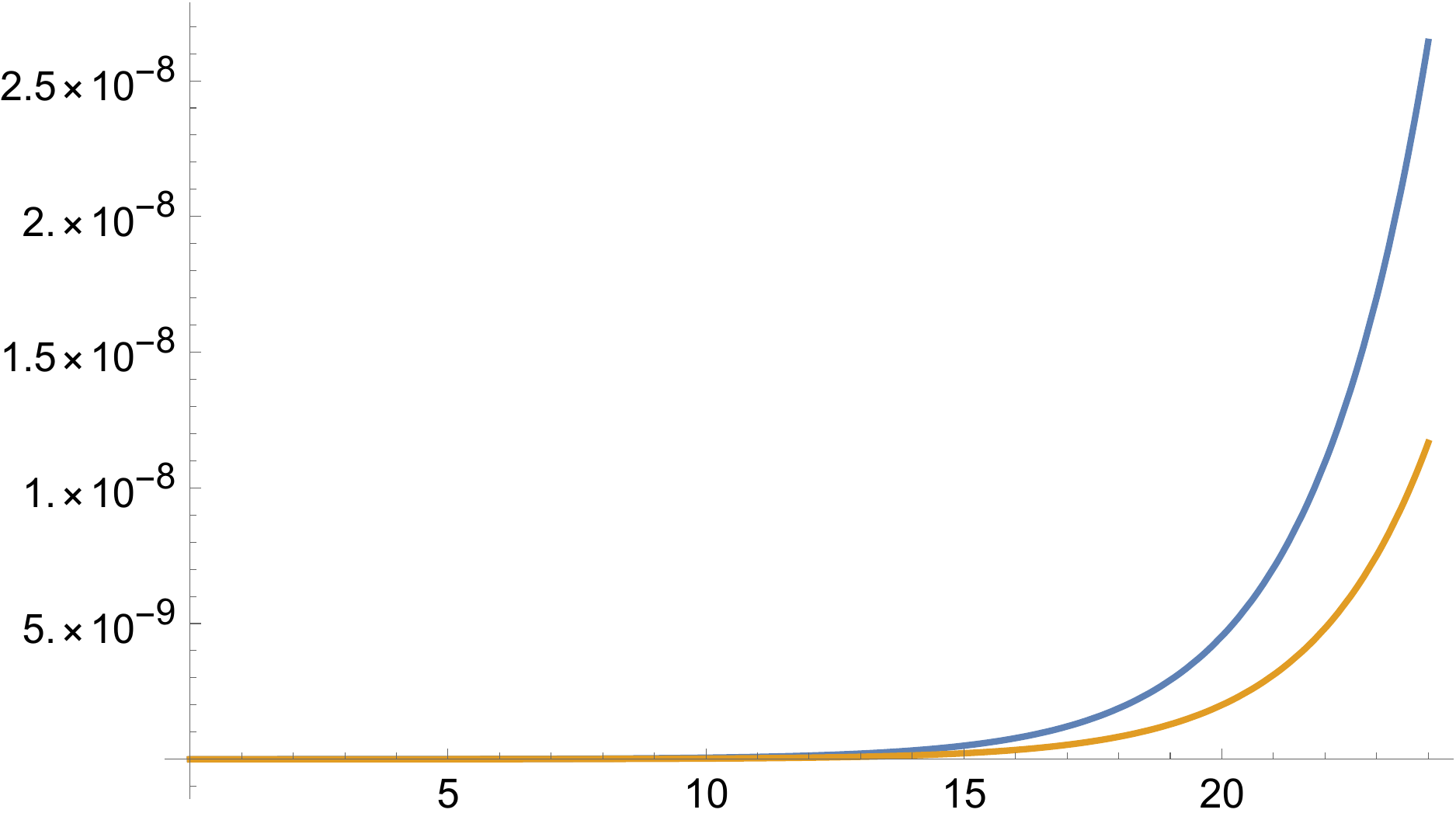}\ \ 
  \includegraphics[width=6.2cm]{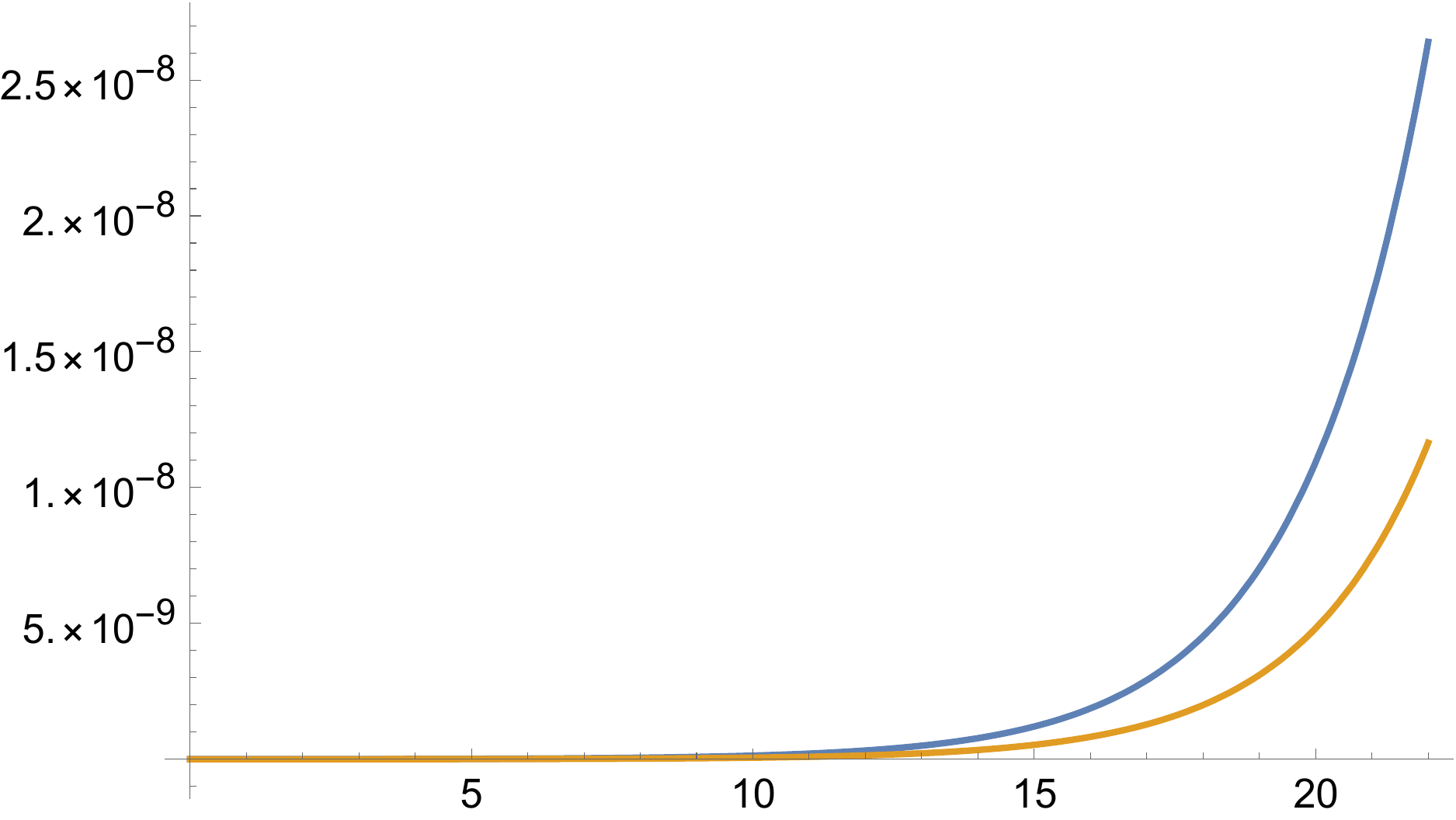}
  \caption{Absolute differences between entry (left) and leaving (right) arcs and the corresponding part of the solution for $T=63$. On the right, the time is reversed (from $Q_l$ to $P$).}
  \label{fig:time-63-diff}
\end{figure}

Finally, Figure \ref{fig:20-times} contains the plots of the different solutions $x(t)$ for times $T$ between $51$ and $56$ and for time $T=63$. The structure of the entry arc is essentially the same for all and all are, obviously, indistinguishable, whereas the leaving arc is also essentially equal but starts at different times. Figure \ref{fig:20-times-diffs} shows the difference between the corresponding entry and leaving arcs and the ones of the turnpike (where the cutting point is set as above).

\begin{figure}[h!]
  \centering
  \includegraphics[width=8cm]{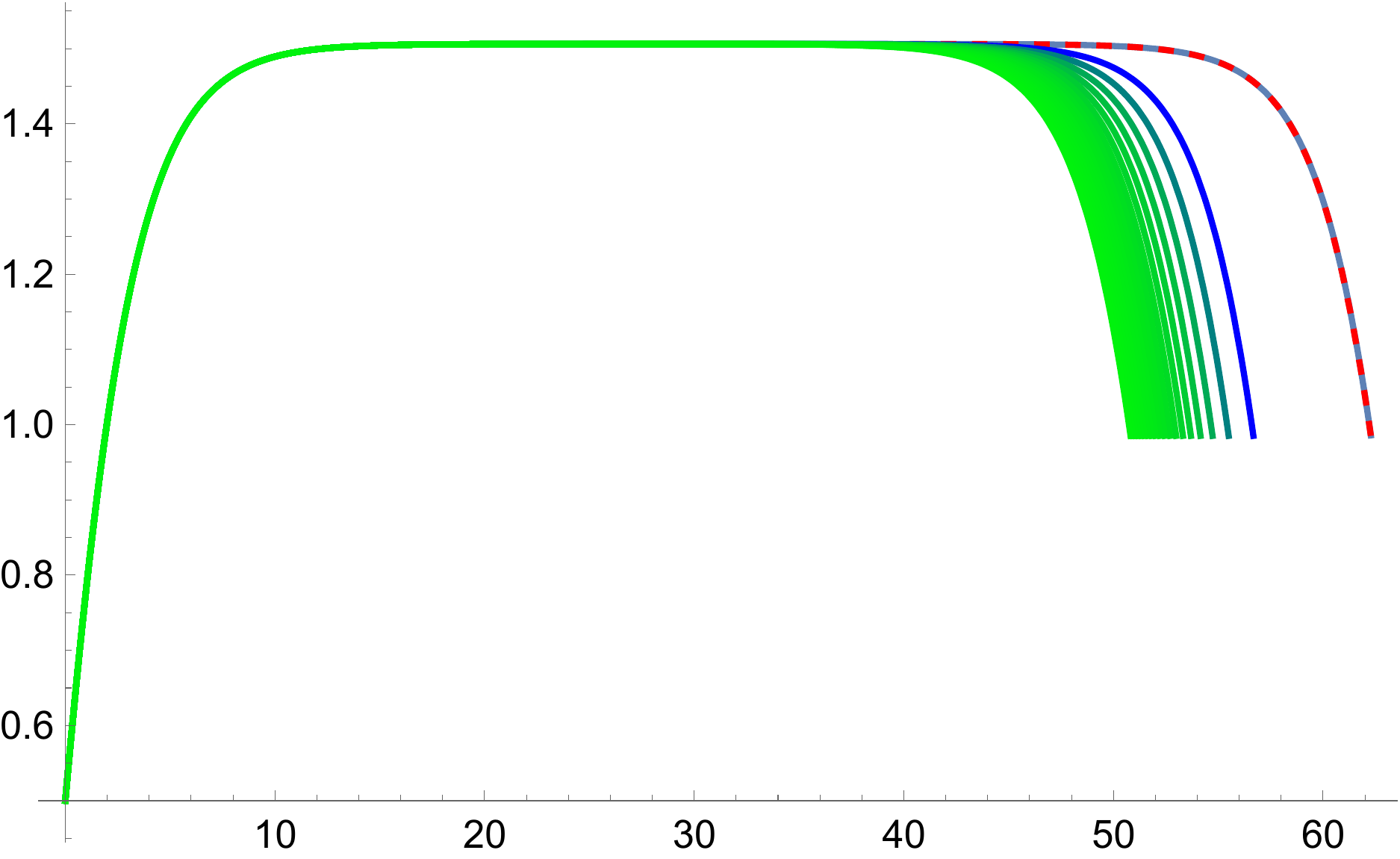}
  \caption{Solutions for times between 51 and 56, and for $T=63$.}
  \label{fig:20-times}
\end{figure}

\begin{figure}[h!]
  \centering
  \includegraphics[width=6.2cm]{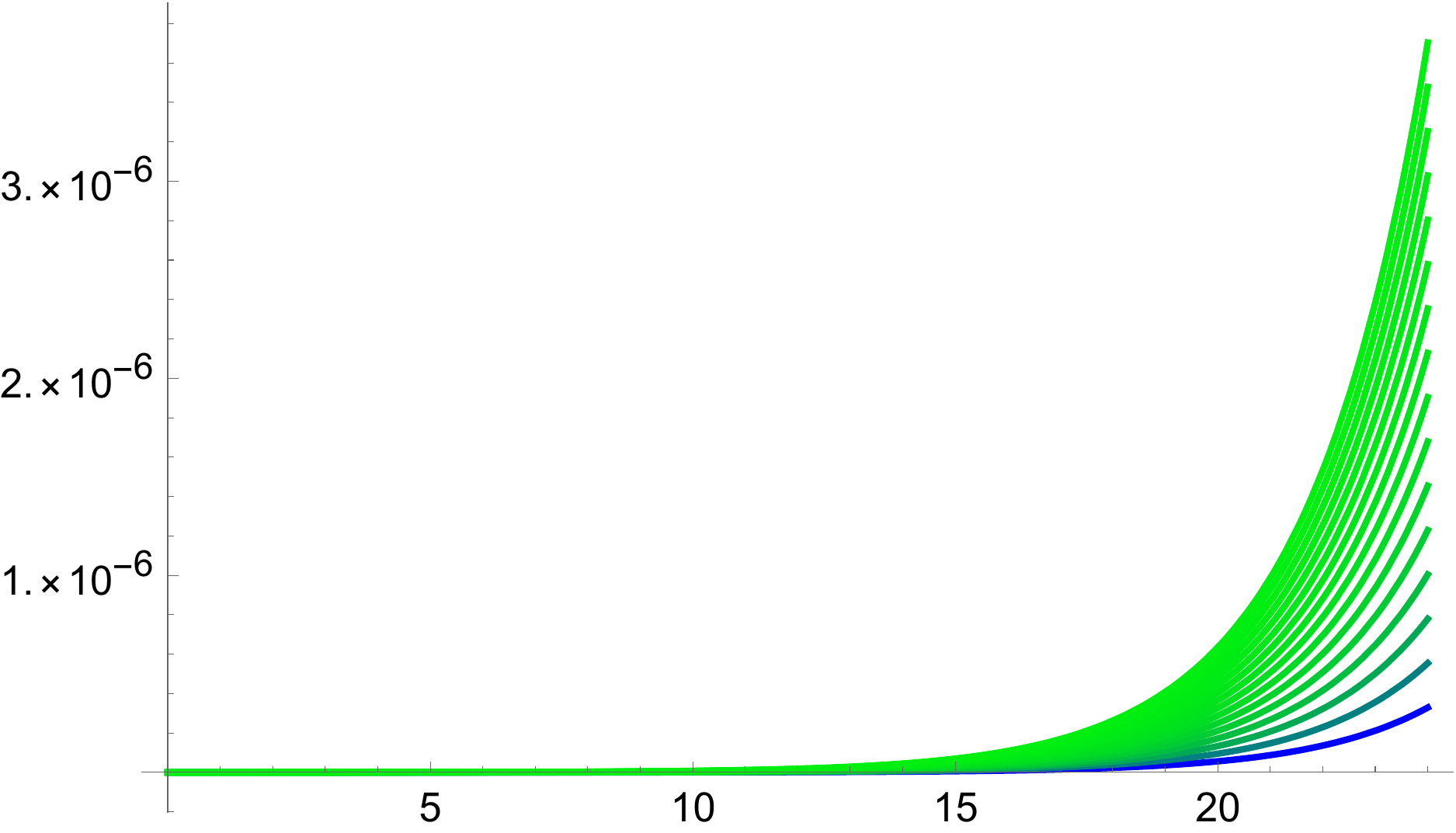}\ \ 
  \includegraphics[width=6.2cm]{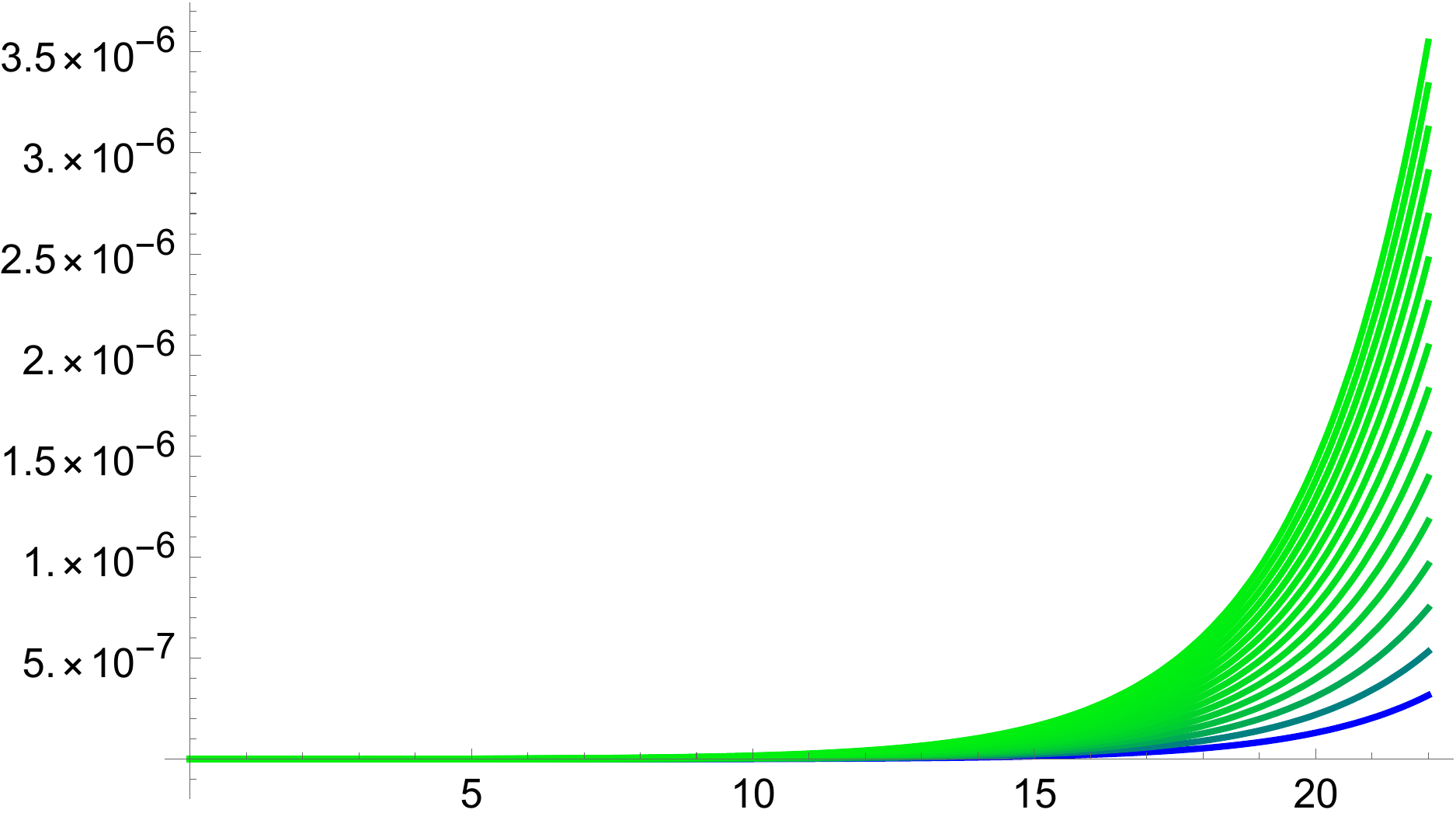}
  \caption{Absolute difference between the solutions in Figure \ref{fig:20-times} and the entry (left) and leaving (right) arcs (time is reversed on the right).}
  \label{fig:20-times-diffs}
\end{figure}

\section{Final remarks}
Our aim in this paper is just to show, in the case of dimension $1$, which is the most graphical one, how to compute the entry and leaving arcs of the turnpike of an autonomous variational problem in order to settle this question. Of course, the generalization to variational problems in which the functional $F(x_1,\dot{x}_{1},\ldots, x_k,\dot{x}_k)$ has ``separated variables'', that is problems with
\begin{equation*}
  \frac{\partial^2 F}{\partial u \partial v} = 0
\end{equation*}
whenever $u, v$ correspond to variables with different indices (i.e. $u\in \left\{ x_i,\dot{x}_{i} \right\}$ and $v\in \left\{ x_j,\dot{x}_j \right\}$ with $i\neq j$) is straightforward, as the associated vector fields are defined by independent equations.

The most general autonomous case is, for the time being, inaccessible to us but we hope the technique presented in this work may be useful to elucidate their solution.

\end{document}